\documentclass[12pt,leqno]{amsart}  
\usepackage{amsmath,amstext,amsthm,amssymb,amsxtra}

\usepackage{txfonts} 
\usepackage[T1]{fontenc}
\usepackage{lmodern}

 \usepackage{euler}    

 \usepackage{mathtools}
\mathtoolsset{showonlyrefs,showmanualtags}

\usepackage[colorlinks,hypertexnames=false]{hyperref}
\usepackage{amsrefs}
\allowdisplaybreaks

\theoremstyle{plain} 
\newtheorem{lemma}[equation]{Lemma} 
\newtheorem{proposition}[equation]{Proposition} 
\newtheorem{theorem}[equation]{Theorem} 
 
\newtheorem*{pz}{Paley-Zygmund Inequality}
\newtheorem*{pz2}{Second Paley-Zygmund Inequality}
\newtheorem*{roth}{K.~Roth's Theorem}
\newtheorem*{schmidt}{Schmidt's Theorem}
\newtheorem*{sbi}{Conjecture: The Small Ball Inequality}
\newtheorem*{ssbi}{Conjecture: The Signed Small Ball Inequality}

\theoremstyle{definition}
\newtheorem{definition}[equation]{Definition} 

\theoremstyle{remark}
\newtheorem{remark}[equation]{Remark}
\newtheorem*{dedication}{Dedication to Walter Philipp}

\numberwithin{equation}{section}

\def\norm#1.#2.{\lVert#1\rVert_{#2}}
\def\Norm#1.#2.{\bigl\lVert#1\bigr\rVert_{#2}}
\def\NOrm#1.#2.{\Bigl\lVert#1\Bigr\rVert_{#2}}
\def\NORm#1.#2.{\biggl\lVert#1\biggr\rVert_{#2}}
\def\NORM#1.#2.{\Biggl\lVert#1\Biggr\rVert_{#2}}

\def\ip#1,#2,{\langle #1,#2\rangle}
\def\Ip#1,#2,{\bigl\langle#1,#2\bigr\rangle}
\def\IP#1,#2,{\Bigl\langle#1,#2\Bigr\rangle}

\def\abs#1{\lvert#1\rvert}

\def\ABs#1{\biggl\lvert#1\biggr\rvert}

\def\XXint#1#2#3{{\setbox0=\hbox{$#1{#2#3}{\int}$}
     \vcenter{\hbox{$#2#3$}}\kern-.5\wd0}}

\begin{document}
\title {A Three Dimensional Signed Small Ball Inequality }

\thanks{The authors are grateful to the Fields Institute, and the American Mathematical Institute
for hospitality and support, and to the National Science Foundation for support through the 
grants DMS-0456538 and DMS-0801036}

\author[D. Bilyk]{Dmitriy Bilyk }
\address {School of Mathematics, Institute for Advanced Study, Princeton, NJ 08540, USA.}
\email{bilyk@math.ias.edu}

\author[M.T. Lacey]{Michael T. Lacey}
\address{ School of Mathematics, Georgia Institute of Technology, Atlanta GA 30332, USA.}
\email{lacey@math.gatech.edu}

\author[I. Parissis]{Ioannis Parissis}
\address{Institutionen f\"or Matematik, Kungliga Tekniska H\"ogskolan, SE 100 44, Stockholm, SWEDEN.}
\email{ioannis.parissis@gmail.com}

\author[A. Vagharshakyan]{Armen Vagharshakyan }
\address{ School of Mathematics, Georgia Institute of Technology, Atlanta GA 30332, USA.}
\email{armenv@math.gatech.edu}

\begin{abstract} Let $ R$ denote dyadic rectangles in the unit cube $ [0,1] ^{3}$ in three dimensions. 
Let $ h_R$ be the $ L ^{\infty }$-normalized Haar function whose support is $ R$.  We show that 
for all integers $ n\ge 1$ and choices of coefficients $a_R\in \{\pm1\} $, we have 
\begin{equation*}
\NOrm \sum _{\substack{\lvert  R\rvert= 2 ^{-n}\\ \lvert  R_1\rvert\ge 2 ^{-n/2} } } a_R \, h_R . L ^{\infty }. \gtrsim n ^{9/8}\,. 
\end{equation*}
The trivial $ L ^{2}$ lower bound is $ n$, and the sharp lower bound would be $ n ^{3/2}$.  This is the best 
exponent known to the authors. This inequality is motivated
by new results on the star-Discrepancy function in all dimensions $ d\ge 3$.  

\end{abstract}

\maketitle

\section{Introduction} 
 
We are motivated by the classical question of irregularities of distribution \cite{MR903025} and recent results 
which give new lower bounds on the star-Discrepancy in all dimensions $ d\ge 3$ \cites{MR2409170,MR2414745}.  We recall these results.

Given integer $ N$, and selection $ \mathcal P$ of $ N$ points in the unit cube $ [0,1] ^{d}$, we define a \emph{Discrepancy 
Function} associated to $ \mathcal P$ as follows.  At any point $ x\in [0,1] ^{d}$, set 
\begin{equation*}
D_N (x) = \sharp ( \mathcal P \cap [0,x))- N \lvert  [0,x)\rvert \,.  
\end{equation*}
Here, by $ [0,x)$ we mean the $ d$-dimensional rectangle with left-hand corner at the origin, and 
right-hand corner at $ x\in [0,1] ^{d}$.  Thus, if we write $ x= (x_1 ,\dotsc, x_d)$ we then have 
\begin{equation*}
[0,x)= \prod _{j=1} ^{d} [0,x_j) \,. 
\end{equation*}
At point $ x$ we are taking the difference between the actual  and 
the expected number of points in the rectangle. 
  Traditionally, the dependence of $ D_N$ 
on the selection of points $ \mathcal P$ is only indicated through the number of points in the collection 
$ \mathcal P$.  We mention only the main points of the subject here, and leave the (interesting)
history of the subject to references such as \cite{MR903025}.  

The result of Klaus Roth  \cite{MR0066435} gives a definitive average case lower bound on the Discrepancy function.  

\begin{roth} 
For any dimension $ d\ge 2$,   we have the following estimate 
\begin{equation} \label{e.roth}
\norm D_N. 2. \gtrsim (\log N) ^{(d-1)/2} \,. 
\end{equation}
\end{roth}

The same lower bound holds in all $ L ^{p}$, $ 1<p<\infty $, as observed by Schmidt \cite{MR0319933}. 
But, the $ L ^{\infty } $ infinity estimate is much harder.  In dimension $ d=2$ the definitive 
result was obtained by Schmidt again \cite{MR0491574}.

\begin{schmidt}  \label{t.schmidt}
We have the estimates below, valid for all collections $\mathcal P\subset [0,1]^2 $:  
\begin{align}  \label{e.schmidt}
\norm D_N .\infty. {}\gtrsim{} \log N . 
\end{align}
\end{schmidt} 

The $ L ^{\infty }$ estimates are referred to as star-Discrepancy bounds.  Extending and greatly 
simplifying an intricate estimate of Jozef Beck \cite{MR1032337}, some of these authors have obtained a partial 
extension of Schmidt's result to all dimensions $ d\ge3$.

\begin{theorem}\label{t.dge3}[\cites{MR2409170,MR2414745}] For dimensions $ d\ge 3$ there is an $ \eta = \eta (d)>0$ for which we have 
the inequality 
\begin{equation*}
\norm D_N. \infty . \gtrsim (\log N) ^{(d-1)/2 + \eta } \,. 
\end{equation*}
That is, there is an $ \eta $ improvement in the Roth exponent. 
\end{theorem}

As explained in these references, the analysis of the star-Discrepancy function is closely related to other 
questions in probability theory, approximation theory, and harmonic analysis.  We turn to one of these, 
the simplest to state question, which is central to all of these issues.  We begin with the definition of the Haar
functions.

In one dimension, the dyadic intervals of the real line $ \mathbb R $ are given by 
\begin{equation*}
\mathcal D = \bigl\{ [j 2 ^{k}, (j+1) 2 ^{k}) \;:\; j,k\in \mathbb Z  \bigr\}\,. 
\end{equation*}
Any interval $ I$ is a union of its left and right halves, denoted by $ I _{\textup{left/right}}$, 
which are also dyadic. 
The \emph{Haar function $ h_I$ associated to $ I$}, or simply \emph{ Haar function} is 
\begin{equation*}
h_I = -\mathbf 1_{I _{\textup{left}}}+ \mathbf 1_{I_{\textup{right}}}
\end{equation*}
 Note that for dyadic intervals $ J\subsetneq I$, the  
Haar function $ h_J$ is completely supported on a set where $ h_I $ is 
constant.  This basic property leads to far-reaching implications that we will exploit in these 
notes.  

In higher dimensions $ d\ge 2$, we take the dyadic rectangles to be the tensor product of dyadic intervals in dimension $ d$:
\begin{equation*}
\mathcal D ^{d} = \bigl\{ R=R_1 \times \cdots \times R_d \;:\; R_1 ,\dotsc, R_d\in \mathcal D \bigr\}\,. 
\end{equation*}
The \emph{Haar function} associated to $ R\in \mathcal D_d$ is likewise defined as 
\begin{equation} \label{e.haarFunctionDef}
h _{R } (x_1 ,\dotsc, x_d) = \prod _{j=1} ^{d} h _{R_j} (x_j)\,, \qquad R= R_1 \times \cdots \times R_d\,. 
\end{equation}
While making these definitions on all of $ \mathbb R ^{d}$, we are mainly interested in local questions, 
thus rectangles $ R \subset [0,1] ^{d}$ are always dyadic rectangles $ R\in \mathcal D ^{d}$. 
Namely, we are mainly interested in the following conjectural \emph{reverse triangle inequality} for sums of Haar functions
on $ L ^{\infty }$: 

\begin{sbi} For dimensions $ d\ge 3$, there is a constant $ C_d$ so that for all integers $ n\ge 1$, and 
constants $ \{ a_R \;:\; \lvert  R\rvert=2 ^{-n}\,,\ R\subset [0,1] ^{d} \}$, we have 
\begin{equation} \label{e.sbi}
n ^{ (d-2)/2} \NORm \sum 
_{\substack{\lvert  R\rvert\ge 2 ^{-n} \\ R\subset [0,1] ^{d} }} 
a_R \cdot h_R . \infty . \ge C_d 2 ^{-n} \sum 
_{\substack{\lvert  R\rvert=2 ^{-n} \\ R\subset [0,1] ^{d} }}  \lvert  a_R\rvert \,.  
\end{equation}
\end{sbi}
We are stating this inequality in its strongest possible form.  On the left, the sum goes over all 
rectangles with volume \emph{at least} $ 2 ^{-n}$, while on the right, we only sum over rectangles with 
volume \emph{equal to} $ 2 ^{-n}$. Given the primitive state of our knowledge of this conjecture, 
we will not insist on this distinction  below.

For the case of $ d=2$,  \eqref{e.sbi} holds, and is a Theorem of Talagrand \cite{MR95k:60049}.
(Also see \cites{MR637361,MR0319933,MR96c:41052}).    

The special case of the Small Ball Inequality when all the coefficients $ a_R$ are equal to either 
$ -1$ or $ +1$ we refer to as the `Signed Small Ball Inequality.' Before stating this conjecture, 
let us note that we have the following (trivial) variant of Roth's Theorem in the Signed case: 
\begin{equation*}
 \NORm \sum 
_{\substack{\lvert  R\rvert= 2 ^{-n} \\ R\subset [0,1] ^{d} }} 
a_R \cdot h_R . \infty .  \gtrsim n ^{ (d-1)/2} \,, \qquad a_R\in \{\pm 1\} \,. 
\end{equation*}
The reader can verify this by noting that the left-hand side can be written as about $ n ^{d-1}$ orthogonal 
functions, by partitioning the unit cube into homothetic copies of dyadic rectangles of a fixed volume.  
The Signed Small Ball Inequality asserts a `square root of $ n$' gain over this average case estimate. 

\begin{ssbi} For coefficients $ a_R \in \{\pm1\}$, 
\begin{equation} \label{e.signed}
\NORm \sum 
_{\substack{\lvert  R\rvert=2 ^{-n} \\ R\subset [0,1] ^{d} }} 
a_R \cdot h_R . \infty . \ge C'_d n ^{ d/2} \,, \qquad \,.   
\end{equation}
Here, $ C'_d$  is a constant that only depends upon dimension.  
\end{ssbi} 

We should emphasize that random selection of the coefficients shows that the power on $ n$ on the right 
is sharp.  Unfortunately, random coefficients are very far from the `hard instances' of the inequality, so 
do not indicate a proof of the conjecture.

The Signed Small Ball Conjecture should be easier, but  even  this special case eludes us. 
To illustrate the difficulty in this question, note that in dimension $ d=2$, each point 
$ x$ in the unit square is in $ n+1$ distinct dyadic rectangles of volume $ 2 ^{-n}$.  Thus, it suffices 
to find a \emph{single} point where all the Haar functions have the same sign.  This we will do explicitly 
in \S~\ref{s.talagrand} below.  

Passing to three dimensions reveals a much harder problem.  Each point  $ x$ in the unit cube is in about $ n ^2 $ 
rectangles of volume $ 2 ^{-n}$, 
but in general we can only achieve a $ n ^{3/2} $ supremum norm.  Thus, the task is to find a single point $ x$ 
where the number of pluses is more than the number of minuses by $ n ^{3/2}$.  In percentage terms this 
represents only a $ n ^{-1/2}$-percent imbalance over equal distribution of signs.  

The main Theorem of this note is Theorem~\ref{t.3} below, 
which gives the best exponent we are aware of in the Signed Small Ball Inequality. The method of proof 
is also the simplest we are aware of.  (In particular, it gives a better result than 
the more complicated argument in \cite{MR2375609}). Perhaps this argument can inspire further progress on this 
intriguing and challenging question. 

The authors thank the anonymous referee whose attention to detail has helped  greater clarity in 
our arguments.

\begin{dedication}  One of us was a PhD student of Walter Philipp, the last of seven students.  Walter  
was very fond of the subject of this note, though the insights he would have into the recent developments 
are lost to us.  As a scientist, he held himself to high standards in all his areas of study.
As a friend, he was faithful, loyal, and took great pleasure in renewing contacts and friendship.  
He is very much missed.  
\end{dedication}

\section{The Two Dimensional Case} \label{s.talagrand}

This next definition is due to Schmidt, refining a definition of Roth.  
Let $\vec r\in \mathbb N^d$ be a partition of $n$, thus $\vec r=(r_1 ,\dotsc, r_d)$, 
where the  $r_j$ are non negative integers and $\abs{ \vec r} \coloneqq \sum _{t=1} ^{d} r_t=n$.  
Denote all such vectors at $ \mathbb H _n$. (`$ \mathbb H $' for `hyperbolic.') 
For vector $ \vec r $, let $ \mathcal R _{\vec r} $ be all dyadic rectangles 
$ R$ such that for each coordinate $ 1\le t \le d$, $ \lvert  R _t\rvert= 2 ^{-r_t} $.

\begin{definition}\label{d.rfunction} 
We call a function $f$ an \emph{$\mathsf r$-function  with parameter $ \vec r$ } if 
\begin{equation}
\label{e.rfunction} f=\sum_{R\in \mathcal R _{\vec r}}
\varepsilon_R\, h_R\,,\qquad \varepsilon_R\in \{\pm1\}\,. 
\end{equation}
We will use $f _{\vec r} $ to denote a generic $\mathsf r$-function.   
A fact used without further comment is that $ f _{\vec r} ^2 \equiv 1$. 
\end{definition}

Note that in the Signed Small Ball Inequality, one is seeking lower bounds on 
sums $ \sum _{\lvert  \vec r\rvert=n } f _{\vec r}$.

\smallskip

There is a trivial proof of the two dimensional Small Ball Inequality. 

\begin{proposition}\label{p.independent}  
The random variables $ f_{(j,n-j)}$, $ 0\le j \le n$ are independent. 
\end{proposition}

\begin{proof}

The sigma-field generated by the functions  $ \{ f _{ (k,n-k)} \;:\; 0\le k < j\} $ 
consists of dyadic rectangles $ S = S_1 \times S_2$ with $ \lvert  S_1\rvert = 2 ^{-j} $ 
and $ \lvert  S_2\rvert= 2 ^{-n} $. 
On each line segment $ S_1 \times \{x_2\}$, $ f _{(j, n-j)}$ takes the values $ \pm 1$ 
in equal measure, so the proof is finished. 

\end{proof}

We then have 

\begin{proposition}\label{p.2^n}  In the case of two dimensions, 
\begin{equation*}
\mathbb P \Bigl(\sum _{k=0} ^{n} f _{ (k,n-k)} = {n+1} \Bigr)= 2 ^{-n-1}
\end{equation*}

\end{proposition}

\begin{proof}
Note that 
\begin{equation*}
\mathbb P \Bigl(\sum _{k=0} ^{n} f _{ (k,n-k)} = {n+1} \Bigr)= 
\mathbb P \bigl( f _{ (k,n-k)}=1\ \forall 0\le k \le n \bigr)= 2 ^{-n-1}\,. 
\end{equation*}
\end{proof}

It is our goal to give a caricature of this arugment in three dimensions.  See \S\ref{s.heuristics} 
for a discussion.

\section{Elementary Lemmas} 

We recall some elementary Lemmas that we will need in our three dimensional proof.  

\begin{pz}
Suppose that $ Z$ is a positive random variable with $ \mathbb E Z  = \mu _1 $, $ \mathbb E Z ^{2}
= \mu ^{2} _{2}$. Then, 
\begin{equation} \label{e.pz}
\mathbb P (   Z  \ge \mu _1/2    )\ge \tfrac 14 \frac {\mu _2  ^2 } { \mu _1 ^{2}} \,. 
\end{equation}
 \end{pz} 

\begin{proof}
\begin{align*}
\mu _1 = \mathbb E Z 
&= \mathbb E Z  \mathbf 1_{ Z  < \mu _1/2} + \mathbb E Z  \mathbf 1_{Z \ge \mu _1/2} 
\\
& \le \mu _1/2 +  \mu _2 \mathbb P (   Z\ge \mu _1/2 ) ^{1/2}  
\end{align*}
Now solve for $ \mathbb P (   Z\ge \mu _1/2 )$.  
\end{proof}

\begin{pz2}
For all $ \rho _1>1$ there is a $ \rho _2>0$ so   that for all random variables $Z  $ which 
satisfy 
\begin{equation} \label{e.pz2}
\mathbb E Z=0\,, \qquad  \norm Z.2. \le \norm Z . 4. \le \rho _1 \norm Z.2.  
\end{equation}
we have the inequality $ \mathbb P (Z > \rho _2 \norm Z.2.)> \rho _2 $. 
\end{pz2}

\begin{proof}
Let $ Z_+ \coloneqq   Z \mathbf 1_{ Z>0} $ and $ Z_- \coloneqq -  Z \mathbf 1_{Z<0}$, so that $ Z= Z_+-Z_- $. 
Note that $ \mathbb E Z=0$ forces $ \mathbb E Z_+=\mathbb E Z_-$.  And, 
\begin{align*}
\sigma _2 ^2  &\coloneqq  \mathbb E Z ^2 = \mathbb E Z_+ ^2 + \mathbb E Z_- ^2 \,, 
\\
\sigma _4 ^{4} & \coloneqq \mathbb E Z ^{4}= \mathbb E Z_+ ^4 + \mathbb E Z_- ^4\,.
\end{align*}

Suppose that the conclusion is not true.  Namely $ \mathbb P (Z > \rho _2 \sigma _2) < \rho _2 $ 
for a very small $ \rho _2$.  It follows that  
\begin{align*}
\mathbb E Z_+  &\le 
\mathbb E Z_+ \mathbf 1_{ Z_+ \leq \rho _2 \sigma _2} + \mathbb E Z_+ \mathbf 1_{ Z_+ > \rho _2 \sigma _2}
\\ & \le 
\rho _2 \sigma _2 + \mathbb P (Z > \rho _2 \sigma _2) ^{1/2} \sigma _2  \le 
2 \rho _2 ^{1/2} \sigma _2\,, 
\end{align*}
for $ \rho _2<1$.  Hence $ \mathbb E  Z_- = \mathbb E Z_+ \le 2 \rho _2 ^{1/2} \sigma _2$. 
It is this condition that we will contradict below.  

\smallskip 

We also have 
\begin{align*}
\mathbb E Z_+ ^2 & \le 
\mathbb E Z_+ ^2 \mathbf 1_ { \{ Z_+  \leq \rho _2 \sigma _2\}} + \mathbb E Z_+ ^2  \mathbf 1_{ \{ Z_+ > \rho _2 \sigma _2\}}
\\ & \le 
\rho _2 ^2 \sigma _2 ^2 + \rho _2 ^{1/2} \sigma _4 ^2 
\\
& \le 2 \rho _2 ^{1/2} \rho _1 ^2 \sigma _2 ^2 \,. 
\end{align*}  
So for $ \rho _2 < (4 \rho _1) ^{-4}$, we have $ \mathbb E Z_+ ^2 \le \tfrac 12 \sigma _2 ^2 $.  

It follows that we have $ \mathbb E Z_- ^2 \ge \tfrac 12 \sigma _2 ^2 $, and $ \mathbb E Z_- ^{4} \le 
\rho _1 \sigma _2 ^{4}$.  So by \eqref{e.pz}, we have 
\begin{equation*}
\mathbb P (Z_-> \rho _3 \sigma _2) > \rho _3
\end{equation*}
where $ \rho _3$ is only a function of $ \rho _1$.  But this contradicts $ \mathbb E Z_-\le 2 \rho _2 ^{1/2} \sigma _2$,
for small $ \rho _2$, so finishes our proof.
\end{proof}

The Paley-Zygmund inequalities require a higher moment, and in application we find it convenient to use 
the Littlewood-Paley inequalities to control this higher moment.  
Let  $  \mathcal F_0, \mathcal F _1 ,\dotsc,  \mathcal F _{T}$ 
a sequence of increasing sigma-fields generated by dyadic intervals, 
and let $ d _{t} $, $ 1\le t\le  T$ be a martingale difference sequence, 
namely $ \mathbb E (d_t \;:\; \mathcal F _{t-1})=0$ for all $ t=1,2 ,\dotsc, T $.  Set $ f= \sum _{t=1} ^{T} d_t $. 
The martingale square function of $ f$ is $ \operatorname S (f) ^2 \coloneqq \sum_{t=1} ^{T} d_t ^2 $. 
The instance of the Littlewood-Paley inequalities we need are: 

\begin{lemma}\label{l.lp}  With the notation above, suppose that we have in addition 
that the distribution of $ d_t $ is conditionally symmetric given $ \mathcal F_{t-1}$.  
By this we mean that on each atom  $ A$ of $ \mathcal F _{t-1}$, the 
 the distribution of $  d_t \mathbf 1_{A}$ is equal to that 
of $ -  d_t \mathbf 1_{A}$.  Then, we have 
\begin{equation} \label{e.LP}
\norm f . 4.  \simeq \norm \operatorname S (f). 4. \,. 
\end{equation}
\end{lemma}

\begin{proof}
The case of the Littlewood-Paley for even integers can be proved by expansion of the integral, an argument 
that goes back many decades, and our assumption of being conditionally symmetric is added just to 
simplify this proof. Thus, 
\begin{align*}
\norm f. 4. ^{4}  = \sum _{ 1\le t_1, t_2, t_3, t_4\le T} \mathbb E \prod _{u=1} ^{4} d _{t_u} \,.  
\end{align*}
We claim that unless the integers $ 1\le t_1, t_2, t_3, t_4\le T$ occur in pairs of 
equal integers, the expectation on the right above is zero.   This claim shows that 
\begin{equation*}
\norm f. 4. ^{4}  = \sum _{ 1\le t_1, t_2\le T} \mathbb E  d _{t_1} ^2 \cdot d _{t_2} ^2  \,.  
\end{equation*}
It is easy to see that this proves the Lemma, namely we would have 
\begin{equation*}
\norm \operatorname S (f). 4. ^{4} \le 
\norm f. 4. ^{4} \le \tfrac {4!} 2  \norm \operatorname S (f). 4. ^{4}  \,. 
\end{equation*}

\smallskip 

Let us suppose $ t_1\le t_2 \le t_3 \le t_4$.  If we have $ t_3$ strictly less than $ t_4$, then 
\begin{equation*}
\mathbb E \prod _{u=1} ^{4} d _{t_u} = 
\mathbb E \prod _{u=1} ^{3} d _{t_u} \cdot \mathbb E (d _{t_4} \;:\; \mathcal F _{t_3})=0\,. 
\end{equation*}
If we have $ t_1 < t_2=t_3=t_4$, then by conditional symmetry, $ \mathbb E (d _{t_2} ^{3} \;:\; \mathcal F_{t_1})=0$, 
and so we have 
\begin{equation*}
\mathbb E \prod _{u=1} ^{4} d _{t_u} = 
\mathbb E  d _{t_1} \cdot \mathbb E (d^3 _{t_2} \;:\; \mathcal F _{t_1})=0\,. 
\end{equation*}
If we have $ t_1<t_2<t_3=t_4$, the conditional symmetry again implies that 
$ \mathbb E (d _{t_2}  \cdot d _{t_3}^{2} \;:\; \mathcal F_{t_1})=0$, so that 
\begin{equation*}
\mathbb E \prod _{u=1} ^{4} d _{t_u} = 
\mathbb E  d _{t_1} \cdot \mathbb E ( d _{t_2} \cdot d _{t_3} ^2    \;:\; \mathcal F _{t_1})=0\,. 
\end{equation*}
Thus, the claim is proved. 

\end{proof}

We finish this section with an elementary, slightly technical,  Lemma. 

\begin{lemma}\label{l.condIndep} Let  $  \mathcal F_0, \mathcal F _1 ,\dotsc, \mathcal F_{q} $ 
a sequence of increasing sigma-fields. Let $ A_1 ,\dotsc, A_q$ be events, with $ A_t \in \mathcal F_t $.
Assume that for some $ 0<\gamma <1$, 
\begin{equation}\label{e.E>} 
\mathbb E \bigl( \mathbf 1_{A_t} \;:\; \mathcal F_{t-1} \bigr)\ge \gamma \,, \qquad 1\le t \le q 
\end{equation}
We then have that 
\begin{equation}\label{e.F>}
\mathbb P \Bigl( \bigcap _{t=1} ^{q} A_t \Bigr)\ge \gamma ^{q} \,. 
\end{equation}
More generally, assume that 
\begin{equation}\label{e.E>>}
\mathbb P \Bigl( \bigcup _{t=1} ^{q} \bigl\{ \mathbb E \bigl( \mathbf 1_{A_t} \;:\; \mathcal F_{t-1} \bigr)\le \gamma \bigr\}  \Bigr)
\le \tfrac 12 \cdot \gamma ^{q}\,. 
\end{equation}
Then, 
\begin{equation}\label{e.F>>}
\mathbb P \Bigl( \bigcap _{t=1} ^{q} A_t \Bigr)\ge \tfrac 12 \cdot \gamma ^{q} \,. 
\end{equation}
\end{lemma}

\begin{proof}
To prove \eqref{e.F>}, note that by assumption \eqref{e.E>}, and backwards induction we have 
\begin{align*}
\mathbb P \Bigl( \bigcap _{t=1} ^{q} A_t \Bigr) & =  
\mathbb E \prod  _{t=1} ^{q} \mathbf 1_{A_t} 
\\
& = \mathbb E \prod  _{t=1} ^{q-1} \mathbf 1_{A_t}  \times \mathbb E \bigl(\mathbf 1_{A_q} \;:\; \mathcal F_{q-1} \bigr)
\\
& \ge \gamma \mathbb E \prod  _{t=1} ^{q-1} \mathbf 1_{A_t} 
\\
& \;\; \vdots 
\\
&\ge \gamma ^{q} \,. 
\end{align*}

\medskip 
To prove \eqref{e.F>>}, let us consider an alternate sequence of events.  Define  
\begin{equation*}
\beta _t \coloneqq  
\bigl\{ \mathbb E ( {A_t} \;:\; \mathcal F_{t-1} )\le \gamma \bigr\} \,. 
\end{equation*}
These are the `bad' events.  Now define $ \widetilde A_t \coloneqq A_t \cup \beta _t $. 
By construction, the sets $ \widetilde A_t$ satisfy \eqref{e.E>}. Hence, we have by \eqref{e.F>}, 
\begin{equation*}
\mathbb P \Bigl( \bigcap _{t=1} ^{q} \widetilde A_t \Bigr)\ge \gamma ^{q} \,. 
\end{equation*}
But, now note that by \eqref{e.E>>}, 
\begin{align*}
\mathbb P \Bigl( \bigcap _{t=1} ^{q} A_t \Bigr)
& = \mathbb P \Bigl( \bigcap _{t=1} ^{q} \widetilde A_t \Bigr) - \mathbb P \Bigl(\bigcup _{t=1} ^{q} \beta _t \Bigr)
\\
& \ge \gamma ^{q}-  \tfrac 12 \cdot \gamma ^{q} \ge  \tfrac 12 \cdot \gamma ^{q} \,. 
\end{align*}

\end{proof}

\section{Conditional Expectation Approach in Three Dimensions} 

This is the main result of this note.

\begin{theorem}\label{t.3} For $ \lvert  a_R\rvert=1 $ for all $ R$, we have the estimate 
\begin{equation*}
\NOrm \sum _{\substack{\lvert  R\rvert= 2 ^{-n}\\ \lvert  R_1\rvert\ge 2 ^{-n/2} } } a_R \, h_R . L ^{\infty }. \gtrsim n ^{9/8}\,. 
\end{equation*}
We restrict the sum to those dyadic rectangles whose first side has the lower bound $ \lvert  R_1\rvert\ge 2 ^{-n/2} $. 
\end{theorem}

Heuristics for our proof are given in the next section.  
The restriction on the first side lengths of the rectangles is natural from the point of view of our proof, 
in which the first coordinate plays a distinguished role.  Namely, if we hold the first side length fixed, 
we want the corresponding sum over $ R$ to be suitably generic.   
Let $ 1\ll q \ll n$ be integers.  The integer $ q $ will be taken to be $ q \simeq  n ^{1/4}$.  Our `gain over 
average case' estimate will be $ \sqrt q \simeq  n ^{1/8}$.  While this is a long way from $ n ^{1/2}$, it is 
much better than the  explicit gain of $ 1/24$ in \cite{MR2375609}.

\smallskip 
We begin the proof. 
Let $ \mathcal F _t$ be the sigma field generated by dyadic intervals in [0,1] with 
$ \lvert  I\rvert = 2 ^{- \lfloor t n/q\rfloor} $, for $ 1\le t \le \tfrac 12 q $.  
Let $ \mathbb I _{t} \coloneqq \{ \vec r \;:\; (t-1)n/q\le r_1 < tn/q\}$. Note that the size $ \# \mathbb I _{t} \approx n^2 / q$.  
Let $ f _{\vec r}$ be the $ \mathsf r$-functions specified by the choice of signs in Theorem~\ref{t.3}.  
Here is a basic observation.  

\begin{proposition}\label{p.condExp}  Let $ I\in \mathcal F _t$. 
The distribution of $ \{ f _{\vec r} \;:\; \vec r\in \mathbb I _t\}$ restricted to the set 
$ I \times [0,1] ^2 $ with normalized Lebesgue measure is that of
\begin{equation*}
\{ f _{\vec s} \;:\;  \lvert  \vec s\rvert=n-\lfloor t n/q\rfloor \,,\, 0\le s_1< n/q\} \,,  
\end{equation*}
where the $ f _{\vec s}$ are some  $ \mathsf r$-functions.
The exact specification of this collection depends upon the atom in $ \mathcal F_t$.  
\end{proposition}

\begin{proof}
An atom $ I$ of $ \mathcal F_t$ are dyadic intervals of length $ 2 ^{-\lfloor t n/q\rfloor}$. 
For $\vec r\in \mathbb I _{t}$, $ f _{\vec r}$ restricted to $ I \times [0,1] ^2 $, with normalized 
measure, is an $ \mathsf r$-function with index 
\begin{equation*}
(r_1- \lfloor t n/q\rfloor,\, r_2, r_3) \,. 
\end{equation*}
The statement holds jointly in $ \vec r\in \mathbb I _t$ so finishes the proof.  
\end{proof}

Define sum of `blocks' of $ f _{\vec r}$ as   
\begin{align} \label{e.B}
B_t &\coloneqq  \sum _{\vec r \in \mathbb I _t } f _{\vec r} \,, 
\\ \label{e.cap}
{\textstyle\bigsqcap}_t & \coloneqq \sum _{\substack{\vec r \neq \vec s\in \mathbb I _t \\ r_1=s_1 }} f _{\vec r} \cdot f _{\vec s} \,. 
\end{align}
The sums ${\textstyle\bigsqcap} _t$ play a distinguished role in our analysis, as revealed by the 
basic computation of a square function in \eqref{e.condSqfn} and the fundamental Lemma~\ref{l.exp}.  
Let us set $ \sigma _t ^2 = \norm B_t. 2 . ^2  \simeq n ^2 /q $, for $ 0\le t \le q/2$.

We want to show that for $ q$ as big as $ c n ^{1/4}$, we have 
\begin{equation}\label{e.>}
\mathbb P \Bigl( \sum _{t=1} ^{q/2} B_t \gtrsim n \sqrt q   \Bigr) > 0 
\end{equation}
In fact, we will show 
\begin{equation*}
\mathbb P \Bigl(\bigcap _{t=1} ^{q/2} \bigl\{B_t \gtrsim n/\sqrt q\bigr\} \Bigr)> 0\,, 
\end{equation*}
 from which \eqref{e.>} follows immediately. 

Note that the event $ \bigl\{B_t \gtrsim n/\sqrt q\bigr\}$ simply requires that $ B_t$ be of typical size, 
and  positive, that is this event will have a large probability.
Clearly, we should try to show that these events are in some sense independent, in which 
case the lower bound in \eqref{e.>} will be of the form $ \operatorname e ^{-C q}$, for some $ C>0$.  
Exact independence, as we had in the two-dimensional case, is too much to hope for. Instead, we will 
aim for some conditional independence, as expressed in Lemma~\ref{l.condIndep}.  
\smallskip 

There is a crucial relationship between $ B_t $ and $ {\textstyle\bigsqcap}_t$, which is expressed 
through the martingale square function of $ B_t$, computed in the first coordinate.  
Namely, define 
\begin{equation}\label{e.SqDef}
\operatorname S (B_t) ^2 \coloneqq 
\sum _{j \in  J_t} \ABs{\sum _{\vec r \;:\; r_1 =j} f _{\vec r}} ^2 
\end{equation}
where $ J_t = \{ s \in \mathbb N \;:\; (t-1)n/q\le s < tn/q\}$. 

\begin{proposition}\label{p.sqfn} We have 
\begin{align}  \label{e.sqfn}
\operatorname S (B_t) ^2  &= \sigma ^2 _t + {\textstyle\bigsqcap}_t \,, 
\\ \label{e.condSqfn}
\operatorname S (B_t \;:\; \mathcal F_t) &=  
\sigma ^2  _t + \mathbb E ({\textstyle\bigsqcap}_t \;:\; \mathcal F_t) \,. 
\end{align}
\end{proposition}
By construction, we have $ {}^{\sharp}\, \mathbb I _t \simeq n ^2 /q$, for $ 0\le t<\tfrac 12 \, q$.

\begin{proof}
In \eqref{e.SqDef}, one expands the square on the right hand side.  Notice that this shows that 
\begin{equation*}
\operatorname S (B_t) ^2  = \sum _{\substack{\lvert  \vec r\rvert= \lvert  \vec s\rvert=n  \\ r_1=s_1 \in J_t }} 
f _{\vec r} \cdot f _{\vec s} \,. 
\end{equation*}
We can have  $ \vec r=\vec s$ for  $ ^\sharp\, \mathbb I _t$ choices of $ \vec r$.  Otherwise, we have a 
term that contributes to  $ {\textstyle\bigsqcap}_t $. The conditional expectation conclusion follows from
\eqref{e.sqfn}
\end{proof}

The next fact is the critical observation in \cites{MR2414745,MR2409170,MR2375609} concerning coincidences, 
assures us that typically on the right in \eqref{e.sqfn}, that the first term $\sigma ^2  _t \simeq n ^2 /q $ 
is much larger than the second $ {\textstyle\bigsqcap}_t$.   See 
\cite{MR2414745}*{4.1, and the discussion afterwords}. 

\begin{lemma}\label{l.exp}  We have the uniform estimate 
\begin{equation*}
\norm {\textstyle\bigsqcap} _t . \operatorname {exp} (L ^{2/3}). \lesssim n ^{3/2} / \sqrt q \,. 
\end{equation*}
Here, we are using standard notation for an exponential Orlicz space. 
\end{lemma}

\begin{remark}\label{r.higher} A variant of Lemma~\ref{l.exp} holds in higher dimensions, which permits 
an extension of Theorem~\ref{t.3} to higher dimensions.  We return to this in \S\ref{s.heuristics}.  
\end{remark}

Let us quantify the relationship between these two observations and our task of proving \eqref{e.>}.  
\begin{proposition}\label{p.tau}  There is a universal constant $ \tau >0$ so that defining the event 
\begin{equation} \label{e.Gamma}
\Gamma _t \coloneqq  \Bigl\{  \mathbb E ({\textstyle\bigsqcap}_t ^2  \;:\; \mathcal F_t)^{1/2}
< \tau n ^2 /q 
\Bigr\}
\end{equation}
we have the estimate 
\begin{equation} \label{e.B>}
\mathbb P ( B_t > \tau \cdot n /\sqrt q \;:\; \Gamma _t ) > \tau \mathbf 1_{\Gamma _t} \,. 
\end{equation}
\end{proposition}

The point of this estimate is that the events $ \Gamma _t$ will be overwhelmingly likely for $ q \ll n$.  

\begin{proof}
This is a consequence of the Paley-Zygmund Inequalities, Proposition~\ref{p.condExp}, 
Littlewood-Paley inequalities, and \eqref{e.condSqfn}. 
Namely, by Proposition~\ref{p.condExp}, we have $ \mathbb E (B_t \;:\; \mathcal F_t)=0$, 
and the conditional distribution of $ B_t$ given $ \mathcal F_t$ is symmetric.  
By \eqref{e.condSqfn}, we have 
\begin{align*}
\mathbb E (B_t ^2  \;:\; \mathcal F_t)= \operatorname S (B_t \;:\; \mathcal F_t)
=\sigma _{t} ^2 + \mathbb E ({\textstyle\bigsqcap}_t \;:\; \mathcal F_t)  \,. 
\end{align*}
We  apply the Littlewood-Paley inequalities \eqref{e.LP} to see that 
\begin{align*}
\mathbb E (B_t ^ 4\;:\; \mathcal F_t) 
& \lesssim   \mathbb E (S (B_t \;:\; \mathcal F_t) ^2 \;:\; \mathcal F_t ) 
\\
& = \sigma _t ^{4} + 2\sigma _t ^2  \mathbb E ({\textstyle\bigsqcap}_t \;:\; \mathcal F_t) 
+ \mathbb E ({\textstyle\bigsqcap}_t ^2 \;:\; \mathcal F_t) \,. 
\end{align*}

The event $ \Gamma _t $ gives an upper bound on the terms involving $ {\textstyle\bigsqcap}_t  $ above. 
This permits us to estimate, as $ \Gamma _t \in \mathcal F_t$, 
\begin{equation*}
\bigl\lvert 
\mathbb E (B_t ^2 \;:\; \Gamma _t ) ^{1/2} - \sigma _t \mathbf 1_{\Gamma _t}
\bigr\rvert \le \tau n/\sqrt q \,, 
\end{equation*}
but $ \sigma _{t} \simeq  n/\sqrt q$, so we have $ \mathbb E (B_t ^2 \;:\; \Gamma _t ) ^{1/2} \simeq n/\sqrt q $. 
Similarly, 
\begin{align*}
\mathbb E (B_t ^4 \;:\; \Gamma _t ) ^{1/4} & \lesssim 
 \sigma _t  +  \sigma _t  ^{1/2} 
 \lvert  \mathbb E ({\textstyle\bigsqcap}_t \;:\; \mathcal F_t) \rvert ^{1/4} 
 +  \mathbb E ({\textstyle\bigsqcap}_t ^2 \;:\; \mathcal F_t)  ^{1/4} 
 \\
 & \lesssim (1+ \tau )\sigma _t  \,. 
\end{align*}
Hence, we can apply the Paley-Zygmund inequality \eqref{e.pz2} to conclude the Proposition. 
\end{proof}

By way of explaining the next steps, let us observe the following.  If  we have 
\begin{equation} \label{e.x}
\mathbb E (\mathbf 1 _{\Gamma _t}    \;:\; \mathcal F_t )\ge \tau \qquad \textup{a.s. } (x_1)\,, 
\qquad  1\le t\le q/2 \,, 
\end{equation}
then  \eqref{e.B>} holds, namely $ \mathbb P ( B_t ( \cdot , x_2, x_3) > \tau \cdot n /\sqrt q \;:\; \Gamma _t ) > \tau $ 
almost surely.  Applying 
Lemma~\ref{l.condIndep}, and in particular \eqref{e.F>}, we then have 
\begin{equation*}
\mathbb P _{x_1} \Bigl( \bigcap _{t=1} ^{q/2} \{ B_t ( \cdot , x_2, x_3)> \tau n/\sqrt q\}    \Bigr) 
\ge \tau ^{q/2} \,. 
\end{equation*}
Of course there is no reason that such a pair $ (x_2,x_3)$ exits.  Still,  
 the second half of Lemma~\ref{l.condIndep} will apply 
if  we can demonstrate that 
we can choose $ x_2, x_3$ so that \eqref{e.x} holds 
except on a set, in the $ x_1$ variable,  of sufficiently small probability.

\smallskip 

Keeping \eqref{e.E>>} in mind, let us identify an exceptional set. Use the sets $ \Gamma _t $ as given 
in \eqref{e.Gamma} to define 
\begin{equation}\label{e.E}
E \coloneqq \Biggl\{ (x_2, x_3) \;:\; 
\mathbb P _{x_1} \Biggl[ 
\bigcup _{t=1} ^{ q/2 }  \Gamma _t ^{c}
\Biggr]  > \operatorname {exp}\bigl(- c_1 (n/ q) ^{1/3}\bigr)
\Biggr\}
\end{equation}
Here, $ c_1>0$ will be a sufficiently small constant, independent of $ n$. 
Let us give an upper bound on this set.  
\begin{align}
\mathbb P _{x_2,x_3} (E) & \le \operatorname {exp}\bigl( c_1 (n/ q) ^{1/3}\bigr)
\cdot 
\mathbb P _{x_1,x_2,x_3} \Bigl(  \bigcup _{t=1} ^{ q/2 }  \Gamma _t ^{c}\Bigr)
\\& 
\le 
\operatorname {exp}\bigl( c_1 (n/ q) ^{1/3}\bigr) 
\sum _{t=1} ^{ q/2 } \mathbb P _{x_1,x_2,x_3} (\Gamma _t ^{c})
\\ & 
\lesssim q \operatorname {exp}\bigl( c_1 (n/ q) ^{1/3}\bigr) \cdot 
\operatorname {exp} \Biggl( - \Bigl[
 \tau (n ^2 /q) 
\norm \mathbb E \bigl( {\textstyle\bigsqcap}_t ^2 
\;:\; \mathcal F_t\bigr) ^{1/2} . \operatorname {exp}(L ^{2/3}) .  ^{-1} 
\Bigr] ^{2/3} \Biggr)
\\ \label{e.c2}
& \le 
q \operatorname {exp} \bigl( (c_1-c_2 \tau ^{2/3}) \cdot (n/ q) ^{1/3} \bigr)
\end{align}
Here, we have used Chebyscheff inequality.  And, more importantly, the convexity of 
conditional expectation and $ L ^2 $-norms to estimate 
\begin{equation*}
\norm \mathbb E \bigl( {\textstyle\bigsqcap}_t ^2 
\;:\; \mathcal F_t\bigr) ^{1/2} . \operatorname {exp}(L ^{2/3}) .
\lesssim n ^{3/2}/\sqrt q \,,
\end{equation*}
by Lemma~\ref{l.exp}.  The implied constant is absolute, and determines the constant $ c_2$ 
in \eqref{e.c2}.  For an absolute choice of $ c_1$, and constant $ \tau '$, we see that we have 
\begin{equation}\label{e.PE}
\mathbb P _{x_2,x_3} (E) \lesssim \operatorname {exp}(-\tau '  (n/q) ^{1/3} ) \,. 
\end{equation}
We only need $ \mathbb P_{x_2,x_3} (E)< \tfrac 12$, but an exponential estimate of this 
type is to be expected.

Our last essential estimate is 

\begin{lemma}\label{l.EB} For  $ 0<\kappa <1$ sufficiently small, 
$ q \le \kappa n ^{1/4}$,  and $ (x_2,x_3)\not\in E$,  we have 
\begin{equation*}
\mathbb P _{x_1} \Bigl( \bigcap _{t=1} ^{q/2} \{ B_t ( \cdot , x_2, x_3)> \tau n/\sqrt q\}    \Bigr) \gtrsim \tau ^{q}  \,.
\end{equation*}
\end{lemma}

Assuming this Lemma, we can select 
$ (x_2,x_3)\not\in E$.  Thus, 
we see that there is some $ (x_1,x_2,x_3) $ so that for all 
$ 1\le t\le q/2$ we have  $ B_t(x_1,x_2,x_3) > \tau n/\sqrt q$, whence 
\begin{equation*}
\sum _{t=1} ^{q/2} B_t(x_1,x_2,x_3) > \tfrac \tau2 \cdot  n \sqrt q \,. 
\end{equation*}
That is, \eqref{e.>} holds.  
And we can make the last expression as big as $ \gtrsim n ^{9/8}$.  

\begin{proof}
If $ (x_2, x_3) \not \in E$, bring together the 
definition of $ E$ in \eqref{e.E},  Proposition~\ref{p.tau}, and Lemma~\ref{l.condIndep}.  
We see that \eqref{e.F>>} holds (with $ \gamma = \tau $, and the $ q$ in \eqref{e.F>>} equal to the current 
$ q/2$) provided 
\begin{equation*}
\tfrac 12 \cdot \tau ^{q/2}> \operatorname {exp}\bigl(- c_1 (n/ q) ^{1/3}\bigr)\,. 
\end{equation*}
But this is true by inspection, for $ q\le \kappa n ^{1/4}$.

\end{proof}

\section{Heuristics} \label{s.heuristics}

In two dimensions, Proposition~\ref{p.2^n} clearly reveals an underlying exponential-square distribution 
governing the Small Ball Inequality.  The average case estimate is $ n ^{1/2}$, and the set on which 
the sum is about $ n$ (a square root gain over the average case) is exponential in $ n$.  

Let us take it for granted that the same phenomena should hold in three dimensions.  Namely, in three dimensions 
the average case estimate for a signed small ball sum is $ n $, then the 
event that the sum exceeds $ n ^{3/2}$ (a square root gain over the average case) is also exponential in $ n$.  
How could this be proved?  Let us write 
\begin{align*}
H & = \sum _{\substack{\lvert  R\rvert= 2 ^{-n}\\ \lvert  R_1\rvert \ge 2 ^{-n}  } } a_R h_R 
= \sum _{\substack{\lvert  \vec r\rvert= n\\ r_1\le n/2 } } f _{\vec r}
= \sum _{j=0} ^{n/2}  \beta _j \,, 
\\
\beta _j & \coloneqq \sum _{\substack{\lvert  \vec r\rvert=n \\ r_1=j }} f _{\vec r} \,. 
\end{align*}
Here we have imposed the same restriction on the first coordinate as we did in Theorem~\ref{t.3}.  With this
restriction, note that each $ \beta _j$ is a two-dimensional sum, hence by Proposition~\ref{p.independent}, 
a sum of bounded independent random variables.  It follows that we have by the usual Central Limit Theorem, 
\begin{equation*}
\mathbb P (\beta _j > c \sqrt n) \ge \tfrac 14 \,, 
\end{equation*}
for a fixed constant $ c$.  If one could argue for some sort of independence of the events $ \{\beta _j>c \sqrt n\}$ 
one could then write 
\begin{align*}
\mathbb P (H>c n ^{3/2}) \ge \mathbb P \Bigl( \bigcap _{j=0} ^{n/2} \{\beta _j>c \sqrt n\} \Bigr) 
\gtrsim \epsilon ^{n} \,, 
\end{align*}
for some $ \epsilon >0$.  This matches the `exponential in $ n$' heuristic. 
We cannot implement this proof for the $ \beta _j$, but can in the more restrictive 
`block sums' used above.  

\bigskip 

We comment on extensions of Theorem~\ref{t.3} to higher dimensions.  Namely, the methods of this paper will 
prove 

\begin{theorem}\label{t.4} For $ \lvert  a_R\rvert=1 $ for all $ R$, we have the estimate estimate in 
dimensions $ d\ge 4$: 
\begin{equation*}
\NOrm \sum _{\substack{\lvert  R\rvert= 2 ^{-n}\\ \lvert  R_1\rvert\ge 2 ^{-n/2} } } a_R \, h_R . L ^{\infty }. 
\gtrsim n ^{ (d-1)/2+ 1/4d}\,. 
\end{equation*}
We restrict the sum to those dyadic rectangles whose first side has the lower bound $ \lvert  R\rvert\ge 2 ^{-n/2} $. 
\end{theorem}

This estimate, when specialized to $ d=3$ is worse than that of Theorem~\ref{t.3} due to the fact that the 
full extension of the critical estimate Lemma~\ref{l.exp} is not known to  hold in dimensions $ d\ge 4$.  
Instead, this estimate is known.   Fix the coefficients $ a_R \in \{\pm1\}$ as in Theorem~\ref{t.4}, 
and let $ f _{\vec r}$ be the corresponding $ \mathsf r$-functions.  For $ 1 \ll q \ll n$, define $ \mathbb I _t$ 
as above, namely $ \{\vec r \;:\; \lvert  \vec r\rvert=n\,,\ r_1\in  J _t \}$.  Define 
$ {\textstyle\bigsqcap}_t$ as in \eqref{e.cap}.  The analog of Lemma~\ref{l.exp} in dimensions $ d\ge 4$ are 

\begin{lemma}\label{l.exp4}  In dimensions $ d\ge 4$ we have the estimate 
\begin{equation*}
\norm {\textstyle\bigsqcap}_t . \operatorname {exp}(L ^{2/(2d-1)}).  \lesssim n ^{(2d-3)/2} /\sqrt q \,. 
\end{equation*}
\end{lemma}

See  \cite{MR2409170}*{Section 5, especially (5.3)}, which proves the estimate above for the case of $ q=1$.  
The details of the proof of Theorem~\ref{t.4} are omitted, since the Theorem is at this moment only 
a curiosity. 

It would be quite interesting to extend Theorem~\ref{t.4} to the case where, say, 
one-half of the coefficients are permitted to be zero.  This result would have implications for 
Kolmogorov entropy of certain Sobolev spaces; as well this case is much more indicative of the 
case of general coefficients $ a_R$.  As far as the authors are aware, there is no straight forward 
extension of this argument to the case of even a small percentage of the $ a_R$ being zero.


\end{document}